\documentclass[12pt]{amsart}
\usepackage[hmargin=2.5cm,vmargin=2.5cm]{geometry}
\usepackage{amsfonts}
\usepackage{amsthm}
\usepackage[english]{algorithm2e}
\usepackage{amsmath}
\usepackage{amssymb}
\usepackage{hyperref}
\usepackage[T1]{fontenc}
\newtheorem{Theorem}{Theorem}[section]
\newtheorem{theorem}[Theorem]{Theorem}
\newtheorem{definition}[Theorem]{Definition}
\newtheorem{remark}[Theorem]{Remark}
\newtheorem{example}[Theorem]{Example}
\newtheorem{lemma}[Theorem]{Lemma}
\newtheorem{proposition}[Theorem]{Proposition}

\title{Cohomologies and linear deformations
	of pre-Jacobi-Jordan algebras
}
\author[S. Attan and N. Oro Djibril]
{ Sylvain Attan and Nabil Oro Djibril}
\address{Sylvain Attan \newline
	D\'{e}partement de Math\'{e}matiques, Universit\'{e} d'Abomey-Calavi
	01 BP 4521, Cotonou 01, B\'{e}nin}
\email{syltane2010@yahoo.fr}
\address{Nabil Oro Djibril \newline
	Institut de Math\'ematiques et de Sciences Physiques, Universit\'{e} d'Abomey-Calavi
	01 BP 613-Oganla, Porto-Novo, B\'{e}nin}
\email{nabil.orodjibril@imsp-uac.org}
\begin{document}
	\maketitle
	\begin{abstract}
		The purpose of this paper is to introduce an algebraic cohomology  theory of left pre-Jacobi-Jordan algebras. We use the cohomological approach to study linear deformations of these algebras. We also introduce the notion of Nijenhuis operators on pre-Jacobi-Jordan algebras and we show that the deformation generated by a Nijenhuis operator is trivial. 
	\end{abstract}
	
	\section{Introduction}
	
	Left (resp. right) pre-Jacobi Jordan algebras are algebras whose anti-associator is left (resp. right) skew-symmetric. This justifies the name " left(resp. right) skew-symmetric algebras" given for these algebras, first studied in \cite{absb}. There is a close relationship between pre-Jacobi-Jordan algebras and Jacobi-Jordan algebras: a pre-Jacobi-Jordan algebra $(A, \cdot)$ gives rise to a Jacobi-Jordan algebra $(A, \ast )$ via the anticommutator
	multiplication, i.e., pre-Jacobi-Jordan algebras are Jacobi-Jordan-admissible algebras. Moreover, it is observed that they form a particular class of Jacobi-Jordan-admissible algebras in \cite{absbmm}, where various constructions of these algebras are discussed, in particular an inductive description is provided.  If A is left
	pre-Jacobi-Jordan algebra, then the algebra defined on the same vector space $A$ with "opposite” multiplication $x\bullet y:=
	y\cdot x$ is a right pre-Jacobi-Jordan algebra and vice-versa. Hence, all the statements for left pre-Jacobi-Jordan algebras have their
	corresponding statements for right pre-Jacobi-Jordan algebras. Thus, we will only consider the left pre-Jacobi-Jordan algebra case that we call pre-Jacobi-Jordan algebra for short.
	
	The deformation theory was first started mainly in geometry and physics. More precisely, in quantum mechanics, the non-commutative associative product of operators  is viewed  as a formal associative deformation of the pointwise product of the algebraic structure on symbols of these operators. Therefore, in several works \cite{mger1,mger2,mger3,mger4}, Murray Gerstenhaber introduced algebraic formal deformations of associative algebras and proved that  this theory is intimately linked to the cohomology of  considered algebras. For further information on the extension of such  approach for other algebraic structures, one may refer to, e.g., \cite{haam1,haam2, meam} and their references.
	
	Even though there are important applications of pre-Jacobi-Jordan algebras, cohomologies of these algebras have not been developed due perhaps to the dfficulty.
	The purpose of this paper is to fill this gap. Precisely, we introduce the cohomology theory of pre-Jacobi-Jordan algebras and give its applications. As Jacobi-Jordan algebtras case \cite{abak}, the complex provided here is defined by two sequences of operators
	$d^i$ and $\delta^i$ which satisfy $d^{p+1}\circ\delta^p=0$ for any integer $p\geq 0$. It is called a zigzag cohomology since it deals with  two types of cochains and two sequences of operators. Next, we study  linear deformations of Jacobi-Jordan algebras and prove that these deformations are in one-to-one correspondence with the elements of the second zigzag cohomology group. This shows that the zigzag cohomology fits perfectly and provides the expected results.

	The outline of the paper is as follows: In the second section, some basic notions
	and concepts used through the paper, namely some properties of pre-Jacobi-Jordan algebras and
	their representations are given.
	In section 3,  we introduce a cohomology theory of pre-Jacobi-Jordan algebras. Next, we compute the first and second cohomology groups and give interpretations of these groups.
 In Section 4, we study linear deformation theory of pre-Jacobi-Jordan algebras and show that the cohomology theory introduced in Section 3 fits. In section 5, we introduce the notion of a Nijenhuis operator on a pre-Jacobi-Jordan algebra and show that it generates a trivial deformation of this pre-Jacobi-Jordan algebra.
 
	Throughout this paper, all vector spaces and algebras are meant over a commutative ground field $\mathbb{K}$.
	\section{Basic results on pre-Jaobi-Jordan algebras}
	In this section, we recall some facts about (pre)-Jacobi-Jordan algebras \cite{abak},\cite{absb},\cite{dbaf}, \cite{pzu}
	and we provide some results about their representations.
	\begin{definition}\cite{dbaf}\label{D2}
		A Jacobi-Jordan algebra is an algebra $(A,\ast)$ such that
		\begin{eqnarray}
			\label{3}
			&& x\ast y=y\ast x \mbox{ ( commutativity), }\nonumber\\
			&& J(x,y,z):=\circlearrowleft_{(x,y,z)}(x\ast y)\ast z)=0 \mbox{ \ \ $\forall  x,y,z\in A $,} \label{JJi}
		\end{eqnarray}
		where $\circlearrowleft_{(x,y,z)}$ is the sum over cyclic permutation of $x,y,z.$ The map 
		$J$ is called the Jacobian  of $(A,\cdot) $.
	\end{definition}
In \cite{dbaf}, it is proved that an algebra  is a commutative algebra such that $x^3 =0,\ \forall x\in A$ if and only if it is a Jacobi-Jordan algebra. Moreover, it is found \cite{absb} that every Jacobi-Jordan algebra $(A,\cdot)$ is a nilpotent Jordan algebra such that $x^3 =0,\ \forall x\in A.$
	\begin{definition}\cite{abak}\cite{pzu}\label{D3}
		A representation of a Jacobi-Jordan algebra $ (A,\ast) $ is a couple $ (V,\rho ) $ where $ V $ is a vector space and $ \rho : A\to gl(V) $ is a linear map such that
		\begin{eqnarray}
			\label{28}
			\rho(x\ast y)=-\rho(x)\circ \rho(y)-\rho(y)\circ \rho(x) \mbox{ \ \ $\forall  x,y\in A $.}
		\end{eqnarray}
	\end{definition}
Let prove the following elementary result.
	\begin{proposition}\label{Pro1}
		Let $\mathcal{A}_1:=(A_1, \ast_1)$ and $\mathcal{A}_2:=(A_2, \ast_2)$ be two Jacobi-Jordan algebras and $f:\mathcal{A}_1\rightarrow \mathcal{A}_2$ be a morphism of Jacobi-Jordan algebras. Then $A_2^f:=(A_2, \rho)$ is a representation of $\mathcal{A}_1$ where
		$\rho(a)b:=f(a)\ast_2 b$ for all $(a,b)\in A_1\times A_2.$
	\end{proposition}
	\begin{proof}For all $ x,y\in A_1 $ and $ b\in A_2 $, we have :\\
		$ \begin{array}{lll}
			\rho(x\ast_1 y)b+\rho(x)\circ\rho(y)b+\rho(y)\circ\rho(x)b\stackrel{(\ref{3})}{=}\biggl(f(x)\ast_2f(y)\biggr)\ast_2b+\biggl(f(y)\ast_2b\biggr)\ast_2f(x)\\
			+\biggl(b\ast_2f(x)\biggr)\ast_2f(y)\stackrel{(\ref{3})}{=}0.
		\end{array} $\\
		Hence, $\rho(x\ast_1 y)=-\rho(x)\rho(y)-\rho(y)\rho(x).$
	\end{proof}
	\begin{example}
		\begin{enumerate}
			\item Let $(A,\ast)$ be a Jacobi-Jordan algebra. Define a left multiplication 
			$ L: A\rightarrow gl(A)$  by $L(x)y:=x\ast y$  for all $x, y \in A$. Then $(A , L)$ is a representation of $(A, \ast),$  called a regular representation.
			\item Let $(A,\ast)$ be a Jacobi-Jordan algebra and $I$ be an ideal of $(A,\ast).$ Then $I$  inherits a structure of representation of $(A,\ast)$ where 
			$\rho(a)b:=a\ast b$ for all $(a,b)\in A\times I.$
		\end{enumerate}
	\end{example} 
	\begin{definition}\label{D1}Let $ (A,\cdot) $ be an algebra.
	The anti-associator of $(A,\cdot) $ is the trilinear map  $ Aasso:A\times A\times A\to A $ defined by :
	\begin{eqnarray}
		\label{1}
		Aasso(x,y,z):= (x\cdot
		y)\cdot z+x\cdot (y\cdot z) \mbox{ \ \ $\forall  x,y,z\in A $.}
	\end{eqnarray}
	$ (A,\cdot) $ is called an anti-associative algebra if 
	\begin{eqnarray}
		\label{2}
		Aasso(x,y,z)=0,& \forall x,y,z\in A. 
	\end{eqnarray}
\end{definition}
Now, let's give the definition of algebras which are fundamental in this work. They are first introduced and called left(right) skew-symmetric algebras in \cite{absb}. They generalize anti-associative algebras as you will see in the following definition.
	\begin{definition}
		\label{D4}
		An algebra $ (A,\cdot) $ is called 
		\begin{enumerate}
			\item a left pre-Jacobi-Jordan algebra if
			\begin{eqnarray}
				\label{6}
				Aasso(x,y,z)+Aasso(y,x,z)=0,& \forall x,y,z\in A,
			\end{eqnarray}
			or equivalently
			\begin{eqnarray}
				(x\cdot y)\cdot z+x\cdot(y\cdot z)+(y\cdot x)\cdot z+y\cdot(x\cdot z)=0
				\mbox{ for all $x,y,z\in A.$} \label{6l}
			\end{eqnarray}		 
			\item a right pre-Jacobi-Jordan algebras if
			\begin{eqnarray}
				\label{19}
				Aasso(x,y,z)+Aasso(x,z,y)=0,& \forall x,y,z\in A,
			\end{eqnarray}
			or equivalently
			\begin{eqnarray}
				(x\cdot y)\cdot z+x\cdot(y\cdot z)+(x\cdot z)\cdot y+x\cdot(z\cdot y)=0
				\mbox{ for all $x,y,z\in A.$}\label{19r}
			\end{eqnarray}
		\end{enumerate}
	\end{definition}
	\begin{proposition}\label{Pro3}
		Let $ (A,\cdot) $ be an algebra and $ \bullet $ be the bilinear map defined by\\
		$$ a\bullet b=b\cdot a  \mbox{ $\forall a,b\in A.$}$$	
		Then, $ (A,\cdot) $ is a left pre-Jacobi-Jordan algebra if and only if $ (A,\bullet) $ is a right pre-Jacobi-Jordan algebra.
	\end{proposition}
	\begin{proof}
		For all $ x, y, z\in A $ we have :
		\begin{eqnarray}
			&&	(x\cdot y)\cdot z+x\cdot( y\cdot z)+(y\cdot x)\cdot z+y\cdot (x\cdot z)\nonumber\\
			&&=z\bullet(y\bullet x)+(z\bullet y)\bullet x+z\bullet(x\bullet y)
			+(z\bullet x)\bullet y.\nonumber
		\end{eqnarray}
		Hence, $$(x\cdot y)\cdot z+x\cdot( y\cdot z)+(y\cdot x)\cdot z+y\cdot (x\cdot z)=0 $$ if and only if
		$$ z\bullet(y\bullet x)+(z\bullet y)\bullet x+z\bullet(x\bullet y)
		+(z\bullet x)\bullet y=0.$$	
	\end{proof}
	\begin{remark}
		The algebra $ (A,\bullet)$  denoted  by $ A^{op}$ is called the opposite algebra of the algebra $ (A,\cdot) $.	
	\end{remark}
Thanks to Proposition \ref{Pro3}, in the sequel, by pre-Jacobi-Jordan algebras, we mean left pre-Jacobi-Jordan algebras. \\

Now, let give some examples of pre-Jacobi-Jordan algebras which will be used in next section.
\begin{example}
	Consider the 2-dimensional algebra $ \mathcal{A}_1=(A,\cdot) $ with a basis $ \{e_{1},e_{2}\} $ where nonzero products are $$e_{1}\cdot e_{1}=e_{2}.  $$
	Then $ \mathcal{A}_1$ is a left pre-Jacobi-Jordan algebra. 	
\end{example}
Let consider also another example.
\begin{example}
	Let $ \mathcal{A}_2=(\mathbb{K}^4,\cdot) $ be the left pre-Jacobi-Jordan algebra defined and denoted by $(\mathcal{A}_4,\cdot_D)$ in $[$\cite{absbmm},
	Example 4.1 $]$. It is defined with respect to a basis $\{e_1, e_2, e_3, e_4 \}$ by 
	$$e_1\cdot e_1:=\frac{1}{2}e_2,\ e_1\cdot e_3:=\frac{5}{9}e_4,\ e_3\cdot e_1:=\frac{4}{9}e_4.$$
\end{example}
In the sequel we refer to these examples of left pre-Jacobi-Jordan algebras and their underlying vectors spaces by $\mathcal{A}_1$ and $\mathcal{A}_2$ respectively. These examples will be used in the sequel to illustrate the low degree cohomology of pre-Jacobi-Jordan algebras
	\begin{proposition}
		\label{Pro2}
		Let $ (A,\cdot) $ be a left pre-Jacobi-Jordan algebra. Then the product given by
		\begin{eqnarray}
			\label{27}
			x\ast y=x\cdot y+y\cdot x,  \mbox{ $\forall x,y\in A$, }\nonumber
		\end{eqnarray}
		defines a Jacobi-Jordan algebra structure on $ A $ called the associated (or sub-adjacent) Jacobi-Jordan algebra of $ (A,\cdot) $ and is denoted by $ A^{C}.$ The algebra $ (A,\cdot) $ is called a compatible left pre-Jacobi-Jordan algebra structure on the Jacobi-Jordan algebra $ A^{C}=(A,\ast) $.
	\end{proposition}
	\begin{proof}
		Let $ x,y,z\in A $, then using (\ref{6l}), we have
		\begin{eqnarray}
			&&	J(x,y,z)=\circlearrowleft_{(x,y,z)}(x\ast y)\ast z\nonumber\\
			&&	=\circlearrowleft_{(x,y,z)}\biggl(x\cdot y)\cdot z+x\cdot(y\cdot z)+(y\cdot x)\cdot z+y\cdot(x\cdot z)=0.\nonumber
		\end{eqnarray} 
	\end{proof}
	\begin{proposition}\label{Pro4}
		Let $ (A,\cdot) $ be a left pre-Jacobi-Jordan algebra and $ (B,\diamond) $ be an associative and cummutative algebra. Define
		\begin{eqnarray}
			\label{20}
			(x\otimes a)\bot (y\otimes b)=(x\cdot y)\otimes(a\diamond b ) 
			\mbox{\,\ $\forall (x\otimes a),(y\otimes b)\in A\otimes B$.}	
		\end{eqnarray}
		Then $ (A\otimes B,\bot) $ is a left pre-Jacobi-Jordan algebra.
	\end{proposition}
	\begin{proof}
		Let $ (x\otimes a),(y\otimes b),(z\otimes c)\in A\otimes B$, then
		\begin{eqnarray}
			&& Aasso_{\bot}(x\otimes a,y\otimes b,z\otimes c)+Aasso_{\bot}(y\otimes b,x\otimes a,z\otimes c)\nonumber\\
			&&=((x\otimes a)\bot (y\otimes b))\bot(z\otimes c)+(x\otimes a)\bot ((y\otimes b)\bot(z\otimes c))+((y\otimes b)\bot (x\otimes a))\bot(z\otimes c)\nonumber\\
			&&+(y\otimes b)\bot ((x\otimes a)\bot(z\otimes c))=((x\cdot y)\cdot z)\otimes((a\diamond b)\diamond)c+(x\cdot(y\cdot z))\otimes(a \diamond(b\diamond c))\nonumber\\
			&&+((y\cdot x)\cdot z)\otimes(b\diamond(a\diamond c))
			+(y\cdot(x\cdot z))\otimes(b\diamond(a\diamond c))
			\nonumber\\
			&&=((x\cdot y)\cdot z+x\cdot(y\cdot z)+(y\cdot x)\cdot z+y\cdot(x\cdot z))\otimes(a\diamond(b\diamond c))=0\nonumber\\
			&&\mbox{( by commutativity and associativity of $\diamond$ and (\ref{6l})).}\nonumber
		\end{eqnarray} 
		Therefore, $ (A\otimes B,\bot) $ is a left pre-Jacobi-Jordan algebra.
	\end{proof}
	\begin{definition}\label{D5}
		Let $ \mathcal{A}=(A,\cdot_A) $ and $\mathcal{B}=(B,\cdot_B) $ be two left pre-Jacobi-Jordan algebras. A linear map $ f : A\to B$ is called left  pre-Jacobi-Jordan algebra morphism if
		\begin{eqnarray}
			\label{16}
			f(x\cdot_A y)=f(x)\cdot_B f(y),&\forall x,y\in A.
		\end{eqnarray}
	\end{definition}
	\begin{definition}\label{D6}
		Let $ (A,\cdot) $ be a left pre-Jacobi-Jordan algebra and $ \lambda\in \mathbb{K} $. Let $ P $ be a linear map satisfying
		\begin{eqnarray}
			\label{11}
			P(x)\cdot P(y)=P\bigg(P(x)\cdot y+x\cdot P(y) +\lambda P(x\cdot y)\bigg),& \forall x,y\in A.
		\end{eqnarray}
		Then, $ P $ is called a Rota-Baxter operator of weight $ \lambda $ on $ A $ and
		$ (A;\cdot, P) $ is called a  left pre-Jacobi-Jordan Rota-Baxter algebra of weight $ \lambda $.	
	\end{definition}
	\begin{definition}\label{D7}
		A representation of a left pre-Jacobi-Jordan algebra $ (A,\cdot) $ on a vector space $ V $
		consists of a pair $ (\rho,\mu) $, where $ \rho,\mu : A\to gl(V) $ are linear maps satisfying:	
		\begin{eqnarray}
			\label{7}
			\rho(x\ast y)+\rho(x)\circ \rho(y)+\rho(y)\circ \rho(x)=0, \mbox{\,\ $\forall x,y\in A $, }
		\end{eqnarray}
		\begin{eqnarray}
			\label{8}
			\mu(x\cdot y)+\mu(y)\circ \mu(x)+\mu(y)\circ\rho(x)+\rho(x)\circ\mu(y)=0, \mbox{\,\ $\forall x,y\in A $, }
		\end{eqnarray}
		where $ x\ast y:=x\cdot y+y\cdot x $.\\
		Observe that, condition (\ref{7}) means that $ (V,\rho) $ is a representation of the  Jacobi-Jordan algebra $A^C=(A,\ast) $.	
	\end{definition}
\begin{proposition}\label{Pro67}
	Let $ \mathcal{A}:=(A,\cdot)$ and $ \mathcal{B}:=(B,\intercal)$ be two left pre-Jacobi-Jordan algebras and $ f : \mathcal{A} \to \mathcal{B}$ be a morphism of left pre-Jacobi-Jordan algebra. Then $ \mathcal{B}_{f}:=(B;\rho,\mu)$ is a representation of $\mathcal{A}$ where $ \rho $ and $ \mu $ are maps defined by 
	\begin{eqnarray}
		\label{34}
		\rho(a)b:=f(a)\intercal b,\ \forall (a,b)\in A\times B, \\
		\label{35}
		\mu(a)b:=b\intercal f(a),\  \forall (a,b)\in A\times B.
	\end{eqnarray}
\end{proposition}
\begin{proof}
	Let $ (x,y,a)\in A\otimes A\otimes B $. Using $f$ is a morphism, we have 	
	\begin{eqnarray}
		&& \rho(x\cdot y+y\cdot x)a+\rho(x)\circ \rho(y)a+\rho(y)\circ \rho(x)a\overset{(\ref{34})}{=}\bigg(f(x)\intercal f(y)\bigg)\intercal a+f(x)\intercal\bigg( f(y)\intercal a\bigg)\nonumber\\
		&&
		+\bigg (f(y)\intercal f(x)\bigg)\intercal a+ f(y)\intercal \bigg( f(x)\intercal a\bigg)\overset{(\ref{1}),(\ref{6})}{=}0.\nonumber
	\end{eqnarray}
	Finally, we compute
	\begin{eqnarray}
		&&\mu(x\cdot y)a+\mu(y)\circ\mu(x)a+\mu(y)\circ\rho(x)a+\rho(x)\circ\mu(y)a\nonumber\\
		&&\overset{(\ref{34}),
			(\ref{35})}{=}a\intercal\bigg(f(x)\intercal f(y)\bigg)
		+\bigg(a\intercal f(x)\bigg)\intercal f(y)+\bigg(f(x)\intercal a\bigg)\intercal f(y)+f(x)\intercal\bigg(a\intercal f(y)\bigg)\nonumber\\
		&&\overset{(\ref{1}),(\ref{6})}{=}0.\nonumber
	\end{eqnarray}
	Therefore, $ \mathcal{B}_{f}:=(B;\rho,\mu)$ is a representation of $\mathcal{A}$. 	
\end{proof}
	\begin{example}
		Define the left multiplication $ L : A\to gl(A) $ and the right multiplication $ R : A\to gl(A)  $ by
		$ L_{x}y=x\cdot y $, and $ R_{x}y=y\cdot x $, for all $ x,y\in A $. Then $ (A,L,R) $ is a representation of $ (A, \cdot) $, called the regular representation.
	\end{example}
	\begin{proposition}\label{Pro13}
		Let $ (V,\rho,\mu) $ be a representation of a left pre-Jacobi-Jordan algebra $ (A, \cdot) $ and $ V^{*}=\{f : V\to \mathbb{K}\} $ be a vector space of linear functions on $ V $. Define
		\begin{eqnarray}
			\label{53}
			(\rho^{*}(x)f)v=f(\rho(x)v),  
			& \forall (f,v,x)\in V^{*}\times V\times A,
		\end{eqnarray}
		\begin{eqnarray}
			\label{54}
			(\mu^{*}(x)f)v=f(\mu(x)v),& \forall (f,v,x)\in V^{*}\times V\times A.  
		\end{eqnarray}
		Suppose that 
		\begin{eqnarray}
			\mu(x)\circ \mu(y)=\mu(y)\circ \mu(x),  \mbox{ $\forall  x,y\in A$ },\label{hp}
		\end{eqnarray}
		then $ (V^{*};\rho^{*},\mu^{*}) $  is  a representation of the left pre-Jacobi-Jordan algebra $ (A, \cdot) $.
	\end{proposition}
	\begin{proof} Let $ x,y \in A$ and $ (f,v)\in V^{*}\otimes V $. Suppose that (\ref{hp}) holds.  Next, we compute
		\begin{eqnarray}
			&&\rho^{*}\bigg((x\cdot y+y\cdot x)f\bigg)v+\bigg(\rho^{*}(x)\circ\rho^{*}(y)f\bigg)v+\bigg(\rho^{*}(y)\circ\rho^{*}(x)f\bigg)v
			\nonumber\\
			&&\overset{(\ref{53})}{=}f\bigg(\rho(x\cdot y+y\cdot x)v\bigg)
			+f\bigg(\rho(y)\circ\rho(x)v\bigg)+f\bigg(\rho(x)\circ\rho(y)v\bigg)\nonumber\\
			&&=f\bigg(\rho(x\cdot y+y\cdot x)v+\rho(y)\circ\rho(x)v+\rho(x)\circ\rho(y)v\bigg)\overset{(\ref{7})}{=}0\nonumber
		\end{eqnarray} 
		and then we have also 
		\begin{eqnarray}
			&&\mu^{*}\bigg((x\cdot y)f\bigg)v+\bigg(\mu^{*}(y)\circ\mu^{*}(x)f\bigg)v+\bigg(\mu^{*}(y)\circ\rho^{*}(x)f\bigg)v+\bigg(\rho^{*}(x)\circ\mu^{*}(y)f\bigg)v\nonumber\\
			&&\overset{(\ref{53}),(\ref{54})}{=}f\bigg(\mu(x\cdot y)v\bigg)
			+f\bigg(\mu(x)\circ\mu(y)v\bigg)+f\bigg(\rho(x)\circ\mu(y)v\bigg)+f\bigg(\mu(y)\circ\rho(x)v\bigg)\nonumber\\
			&&=f\bigg(\mu(x\cdot y)v+\mu(x)\circ\mu(y)v
			+\rho(x)\circ\mu(y)v+\mu(y)\circ\rho(x)\bigg)\nonumber\\
			&&\overset{\ref{hp}}{=}f\bigg(\mu(x\cdot y)v+\mu(y)\circ\mu(x)v
			+\rho(x)\circ\mu(y)v+\mu(y)\circ\rho(x)\bigg)\overset{(\ref{8})}{=}0.\nonumber
		\end{eqnarray}
		Hence, $ (V^{*};\rho^{*},\mu^{*}) $  is  a representation of $ (A, \cdot) $.	
	\end{proof}
	\begin{proposition}\label{Pro5}
		Let $(A,\cdot)$  be a left pre-Jacobi-Jordan algebra, $V$ be a vector space and $\rho, \mu : A\rightarrow gl(V)$  be linear maps. Then, $(V,\rho,\mu) $ is  a representation of $ (A, \cdot) $ if and only if $ (A\oplus V,\circledast) $ is a left pre-Jacobi-Jordan algebra called the semi-direct product of $(A,\cdot$) by $(V,\rho,\mu)$ denoted by $ A\ltimes V $ where
		\begin{eqnarray}
			\label{9}
			(x,u)\circledast (y,v):=\left( x\cdot y,\rho(x)v+\mu(y)u\right), \mbox{\,\ $\forall  x,y\in A $, and $ u,v\in V $ }.
		\end{eqnarray}
	\end{proposition}
	\begin{proof}
		Let $ a=(x,u) $, $ b=(y,v) $ and $ c=(z,w) $ be in $ A\oplus V $. Then,
		\begin{eqnarray}
			&&Aasso_{\circledast}(a,b,c)+Aasso_{\circledast}(b,a,c)=(a\circledast b)\circledast c+a\circledast(b\circledast c)
			+(b\circledast a)\circledast c+b\circledast(a\circledast c)\nonumber\\
			&&
			\overset{(\ref{9})}{=}((x\cdot y)\cdot z,\rho(x\cdot y)w+\mu(z)\circ\rho(x)v+\mu(z)\circ\mu(y)u)\nonumber\\
			&&+\left( (y\cdot x)\cdot z,\rho(y\cdot x)w
			+\mu(z)\circ\rho(y)u+\mu(z)\circ\mu(x)v\right)\nonumber\\
			&&+(x\cdot (y\cdot z),\rho(x)\circ\rho(y)w+\rho(x)\circ\mu(z)v+\mu(y\cdot z)u)\nonumber\\
			&&+(y\cdot (x\cdot z),\rho(y)\circ\rho(x)w+\rho(y)\circ\mu(z)u+\mu(x\cdot z)v)\nonumber\\
			&&	=((x\cdot y)\cdot z+x\cdot (y\cdot z)+(y\cdot x)\cdot z+y\cdot (x\cdot z),\rho(x\cdot y)w+\rho(y\cdot x)w+\rho(x)\circ\rho(y)w\nonumber\\
			&&
			+\rho(y)\circ\rho(x)w+\mu(y\cdot z)u+\mu(z)\circ\mu(y)u+
			\mu(z)\circ\rho(y)u+\rho(y)\circ\mu(z)u+\mu(x\cdot z)v\nonumber\\
			&&
			+\mu(z)\circ\mu(x)v+\mu(z)\circ\rho(x)v+\rho(x)\circ\mu(z)v)
			\overset{(\ref{1})}{=}(Aasso(x,y,z)+Aasso(y,x,z),\nonumber\\
			&&\rho(x\ast y)w+\rho(x)\circ\rho(y)w+\rho(y)\circ\rho(x)w+\mu(y\cdot z)u
			+\mu(z)\circ\mu(y)u+\mu(z)\circ\rho(y)u\nonumber\\
			&&+\rho(y)\circ\mu(z)u+\mu(x\cdot z)v+\mu(z)\circ\mu(x)v+\mu(z)\circ\rho(x)v+\rho(x)\circ\mu(z)v).\nonumber
		\end{eqnarray} 
		Hence by (\ref{6}) in $(A,\cdot)$, the identity (\ref{6})  holds in $(A\oplus V,\circledast)$ if and only if (\ref{7}) and (\ref{8}) hold for $(V,\rho,\mu).$
	\end{proof}
	\begin{proposition}\label{Pro6}
		Let $ (V,\rho,\mu) $ be a representation of a left pre-Jacobi-Jordan algebra $ (A,\cdot)$. Then $ (V,\rho+\mu) $ is a representation of the sub-adjacent Jacobi-Jordan algebra $ A^C=(A,\ast) $.
	\end{proposition}
	\begin{proof}
		Let $ x,y\in A $, then  we have			
		\begin{eqnarray}
			&&(\rho+\mu)(x\ast y)+(\rho+\mu)(x)\circ (\rho+\mu)(y)+(\rho+\mu)(y)\circ(\rho+\mu)(x)=\rho(x\ast y)+\mu(x\ast y)\nonumber\\
			&&
			+\rho(x)\circ\rho(y)+\rho(x)\circ\mu(y)+\mu(x)\circ\rho(y)+\mu(x)\circ\mu(y)+\rho(y)\circ\rho(x)+\rho(y)\circ\mu(x)\nonumber\\
			&&
			+\mu(y)\circ\rho(x)+\mu(y)\circ\mu(x)
			=\rho(x\ast y)+\rho(x)\circ\rho(y)+\rho(y)\circ\rho(x)+\mu(x\cdot y)\nonumber\\
			&&+\mu(y)\circ\mu(x)+\mu(y)\circ\rho(x)+\rho(x)\circ\mu(y)
			+\mu(y\cdot x)+\mu(x)\circ\mu(y)+\mu(x)\circ\rho(y)\nonumber\\
			&&+\rho(y)\circ\mu(x)\overset{(\ref{7}),(\ref{8})}{=}0.\nonumber
		\end{eqnarray}
		Therefore $ (V,\rho+\mu) $ is a representation of $(A,\ast) $. 
	\end{proof}
	Next, we introduce  
	derivations (respectively antiderivations) of pre-Jacobi-Jordan algebras with values in  representations.
	\begin{definition} \label{repder}
		Let $(A,\cdot)$ be a pre-Jacobi-Jordan algebra, $(V,\rho,\mu)$ be a representation of $(A,\cdot).$ A linear map $D: A\rightarrow V$ is called 
		\begin{enumerate}
			\item a derivation of $(A,\cdot)$ with values in $(V,\rho,\mu)$ if
			\begin{eqnarray}
				D(u\cdot v)=\mu(v)D(u)+\rho(u)D(v) \mbox{ $\forall u,v\in A,$}\label{vder1}
			\end{eqnarray}
			\item an antiderivation of $(A,\cdot)$  with values in $(V,\rho,\mu)$ if
			\begin{eqnarray}
				D(u\cdot v)=-\mu(v)D(u)-\rho(u)D(v) \mbox{ $\forall u,v\in A.$}\label{vader1}
			\end{eqnarray}
		\end{enumerate}	
	\end{definition}
	Denote by $Der(A,V)$ (resp. $ADer(A,V)$), the set of derivations (resp. antiderivations) of  $(A,\cdot)$ with values in $(V,\rho,\mu).$ 
	\begin{definition}
		Let $(A,\cdot)$ be a left pre-Jacobi-Jordan algebra et $(V,\rho,\mu)$ be a representation of $(A,\cdot).$ Then, the subspace
		$$V^{r\cdot Aas}:=\{v\in V,\rho(x\cdot y)v+\rho(x)\rho(y)v=0, \forall (x,y)\in A^{\otimes 2}\}$$
		$$(\mbox{ resp. $V^{l\cdot Aas}:=\{v\in V,\mu(x\cdot y)v+\mu(y)\mu(x)v=0, \forall (x,y)\in A^{\otimes 2}\}$} )$$
		is called a right (resp. left) anti-associative invariant subspace of $(V,\rho,\mu)$,
		$$V^{r\cdot inv}:=\{v\in V,\rho(x)v=0, \forall x\in A\} \mbox{( resp. $V^{l\cdot inv}:=\{v\in V,\mu(x)v=0, \forall x\in A\} )$ }$$ is called a right (resp. left)  invariant  subspace of $(V,\rho,\mu)$ and
		$$V^{inv}:=\{v\in V,\rho(x)v+\mu(x)v=0, \forall x\in A\}$$
		is called an  invariant  subspace of $(V,\rho,\mu).$ If $V=A$ is the regular representation, then $A^{l\cdot Aas}$ is called a right anti-associative center, $A^{r\cdot inv}$
		(resp. $A^{l\cdot inv}$) is called a right(resp. left) center and $A^{ inv}$ is called a center of $(A,\cdot)$.
	\end{definition}
	\begin{lemma}
		Let $(A,\cdot)$ be a left pre-Jacobi-Jordan algebra and $(V,\rho,\mu)$ be a representation of $(A,\cdot).$ Then,
		\begin{enumerate}
			\item $V^{l\cdot inv}$ (or $V^{r\cdot inv}$) is closed under $\mu$ if and if it is also closed under $\rho.$
			\item $V^{l\cdot inv}\subset V^{l\cdot Aas}$ and  $V^{r\cdot inv}\subset V^{r\cdot Aas}.$
		\end{enumerate}	
	\end{lemma}
	\begin{proof}
		Straightforward computations.
	\end{proof}
	\begin{proposition}\label{inderAV}
		For any $w\in V^{r.Aas}$ i.e., $w\in V$ such that
		\begin{eqnarray}
			\rho(u\cdot v)w+\rho(u)\rho(v)w=0 \mbox{ for all $(u,v)\in A^{\otimes 2} $},\label{vid1}
		\end{eqnarray}
		define 
		$D_w: A\rightarrow V $ by
		\begin{eqnarray}
			D_w(u):=\rho(u)w+\mu(u)w, \mbox{    $\forall u\in A$}.\label{vid2}
		\end{eqnarray}
		Then $D_w$ is an antiderivation of $(A,\cdot)$ with values in $(V,\rho,\mu)$ called an inner antiderivation of $(A,\cdot)$ with values in $(V,\rho,\mu).$
	\end{proposition}
	\begin{proof}
		Let $w\in V^{r.Aas}.$ Then $D_w$ is a linear map.
		Pick $u,v\in A,$ then,
		\begin{eqnarray}
			&&	D_w(u\cdot v)=\rho(u\cdot v)w+\mu(u\cdot v)w=-\rho(u)\rho(v)w+\mu(u\cdot v)w \mbox{ ( by (\ref{vid1}) )}\nonumber\\
			&&=-\rho(u)\rho(v)w-\mu(v)\mu(u)w-
			\mu(v)\rho(u)w-\rho(u)\mu(v)w \mbox{ ( by (\ref{8}) )}\nonumber\\
			&&=
			-\rho(u)(\rho(v)w+\mu(v)w) -\mu(v)(\rho(u)w+\mu(u)w)=-\rho(u)D_w(v)-\mu(v)D_w(u).\nonumber
		\end{eqnarray}
		Therefore, $D_w$ is an antiderivation of  $(A,\cdot)$ with values in $(V,\rho,\mu)$. 
	\end{proof}	
	Denote by $IADer(A,V)$, the set of inner antiderivations of  $(A,\cdot)$ with values in $(V,\rho,\mu).$ Then, 
	$$IADer(A,V):=\{\mu(-)w+\rho(-)w,\ w\in V^{r.Aas}\}.$$ 
	\\	
	The notion of derivations (respectively antiderivations) of  $(A,\cdot)$ with values in the adjoint representation $(A,L,R)$ is called  derivations (respectively antiderivations)   $(A,\cdot)$ as we will see in the following definition.
	\begin{definition} \label{DfDer}
		Let $(A,\cdot)$ be a pre-Jacobi-Jordan algebra.  A linear map $D: A\rightarrow A$ is called 
		\begin{enumerate}
			\item a derivation of $(A,\cdot)$  if
			\begin{eqnarray}
				D(u\cdot v)=D(u)\cdot v+u\cdot D(v) \mbox{ $\forall u,v\in A,$}\label{der1}
			\end{eqnarray}
			\item an antiderivation of $(A,\cdot)$  if
			\begin{eqnarray}
				D(u\cdot v)=-D(u)\cdot v-u\cdot D(v) \mbox{ $\forall u,v\in A.$}\label{ader1}
			\end{eqnarray}
		\end{enumerate}	
	\end{definition}
	Denote by $Der(A)$ (resp. $ADer(A)$) the set of derivations (resp. antiderivations)  of a  pre-Jacobi-Jordan algebra $(A,\cdot).$\\  
	
	Now, Proposition \ref{inderAV} reads as:
	\begin{proposition}
		For any $w\in A^{r.Aas}$, i.e., $w\in A$ such that
		\begin{eqnarray}
			(u\cdot v)\cdot w+u\cdot (v\cdot w)=0 \mbox{ for all $(u,v)\in A^{\otimes 2} $},\label{id1}
		\end{eqnarray}
		define 
		$D_w: A\rightarrow A $ by
		\begin{eqnarray}
			D_w(u)&:=&R_w(u)+L_w(u)
			=u\cdot w +w\cdot u, \mbox{    $\forall u\in A$}.\label{id2}
		\end{eqnarray}
		Then $D_w$ is an antiderivation  called an inner antiderivation of $(A,\cdot).$
	\end{proposition}
	\begin{proof}
		Let $w\in A^{r.Aas}.$ Then $D_w$ is a linear map.
		Pick $u,v\in A.$ Then,
		\begin{eqnarray}
			&&	D_w(u\cdot v)=(u\cdot v)\cdot w+w\cdot (u\cdot v)=-u\cdot(v\cdot w)+w\cdot (u\cdot v) \mbox{ ( by (\ref{id1}) )}\nonumber\\
			&&=-u\cdot(v\cdot w)-(w\cdot u)\cdot v-(u\cdot w)\cdot v-u\cdot (w\cdot v) \mbox{ ( by (\ref{6l}) )}\nonumber\\
			&&= -(u\cdot w+w\cdot u)\cdot v-u\cdot (v\cdot w+w\cdot v)=-D_w(u)\cdot v-u\cdot D_w(v).\nonumber
		\end{eqnarray}
		Hence, $D_w $ is an antiderivation of $(A,\cdot)$. 
	\end{proof}
	Denote by $IADer(A)$, the set of inner antiderivations of $(A,\cdot).$  Then we have:
	$$IADer(A):=
	\{(L_w+R_w,\ w\in A^{r.Aas}\}.$$
	For all $D_1\in Der(A)\cup ADer(A)$ and $D_2\in Der(A)\cup ADer(A)$ define their  commutator by
	$$[D_1,D_2]:=D_1D_2-D_2D_1.$$
	Then, we obtain the following results.
	\begin{lemma}
		Let $(A,\cdot)$ be a pre-Jacobi-Jordan algebra. Then,
		\begin{enumerate}
			\item $[Der(A),Der(A)]\subset Der(A),$
			\item $[ADer(A),ADer(A)]\subset Der(A),$
			\item $[ADer(A),Der(A)]\subset ADer(A).$
		\end{enumerate}
	\end{lemma}
	\begin{proof}
		\begin{enumerate}
			\item  Let $ D_{1},D_{2}\in Der(A) $. For all $ u,v\in A $, we have\\\\
			$ \begin{array}{lll}
				[D_{1},D_{2}](u\cdot v)=(D_{1}\circ D_{2})(u\cdot v)-(D_{2}\circ D_{1})(u\cdot v)\\
				\overset{(\ref{der1})}{=}D_{1}(D_{2}(u))\cdot v+(D_{2}(u))\cdot(D_{1}(v))+(D_{1}(u))\cdot(D_{2}(v))+u\cdot D_{1}(D_{2}(v))-D_{2}(D_{1}(u))\cdot v\\
				-(D_{1}(u))\cdot(D_{2}(v))-(D_{2}(u))\cdot(D_{1}(v))-u\cdot D_{2}(D_{1}(v))\\
				=D_{1}(D_{2}(u))\cdot v-D_{2}(D_{1}(u))\cdot v+u\cdot D_{1}(D_{2}(v))-u\cdot D_{2}(D_{1}(v))\\
				=([D_{1},D_{2}](u))\cdot v+u\cdot([D_{1},D_{2}](v))
			\end{array} $.\\ 
			Hence, $[Der(A),Der(A)]\subset Der(A)$.
			\item Let $ D_{1},D_{2}\in ADer(A) $. We have, \\
			$ \begin{array}{lll}
				[D_{1},D_{2}](u\cdot v)=(D_{1}\circ D_{2})(u\cdot v)-(D_{2}\circ D_{1})(u\cdot v)\\
				\overset{(\ref{ader1})}{=}D_{1}(D_{2}(u))\cdot v+(D_{2}(u))\cdot(D_{1}(v))+(D_{1}(u))\cdot(D_{2}(v))+u\cdot D_{1}(D_{2}(v))-D_{2}(D_{1}(u))\cdot v\\
				-(D_{1}(u))\cdot(D_{2}(v))-(D_{2}(u))\cdot(D_{1}(v))-u\cdot D_{2}(D_{1}(v))\\
				=D_{1}(D_{2}(u))\cdot v-D_{2}(D_{1}(u))\cdot v+u\cdot D_{1}(D_{2}(v))-u\cdot D_{2}(D_{1}(v))\\
				=([D_{1},D_{2}](u))\cdot v+u\cdot([D_{1},D_{2}](v))	
			\end{array} $. \\
			Hence, $[ADer(A),ADer(A)]\subset Der(A)$.
			\item Let $ D_{1}\in ADer(A) $, $ D_{2}\in Der(A) $. We have, \\
			$ \begin{array}{lll}
				[D_{1},D_{2}](u\cdot v)=(D_{1}\circ D_{2})(u\cdot v)-(D_{2}\circ D_{1})(u\cdot v)\\
				\overset{(\ref{der1}),(\ref{ader1})}{=}-D_{1}(D_{2}(u))\cdot v-(D_{2}(u))\cdot(D_{1}(v))-(D_{1}(u))\cdot(D_{2}(v))-u\cdot D_{1}(D_{2}(v))+D_{2}(D_{1}(u))\cdot v\\
				+(D_{1}(u))\cdot(D_{2}(v))+(D_{2}(u))\cdot(D_{1}(v))+u\cdot D_{2}(D_{1}(v))\\
				=-D_{1}(D_{2}(u))\cdot v+-D_{2}(D_{1}(u))\cdot v-u\cdot(D_{1}(D_{2}(v)))+u\cdot(D_{2}(D_{1}(v)))\\
				=-([D_{1},D_{2}](u))\cdot v-u\cdot([D_{1},D_{2}](v))
			\end{array} $\\
			Hense, $[ADer(A),Der(A)]\subset ADer(A).$   
		\end{enumerate}
	\end{proof}
	It is easy to prove:
	\begin{proposition}
		$(Der(A),[,])$ is a Lie algebra.
	\end{proposition}
	Now, let $D_1\in Der(A)\cup ADer(A)$ and $D_2\in Der(A)\cup ADer(A)$ define their anti-commutator by
	$$\{D_1,D_2\}:=D_1D_2+D_2D_1.$$
	In the following two propositions, we give a necessary and sufficient condition for the anti-commutator of two derivations( respectively antiderivations) to be an derivation( respectively antiderivation).
	\begin{proposition} \label{antantid}
		Let $(A,\cdot)$ be a pre-Jacobi-Jordan algebra.
		\begin{enumerate}
			\item  Let $D_1\in ADer(A)$ and $D_2\in ADer(A).$ Then,
			\begin{eqnarray}
				&&	\mbox{$\bullet$ $\{D_1,D_2\}\in ADer(A)$ if and only if}\nonumber\\ &&-2\bigg(\{D_{1},D_{2}\}(u)\bigg)\cdot v-2u\cdot\bigg(\{D_{1},D_{2}\}(v)\bigg)+2(D_{1}(u))\cdot(D_{2}(v))\nonumber\\
				&&+2(D_{2}(u))\cdot(D_{1}(v))=0,\label{antider}\\
				&&	\mbox{$\bullet$ $\{D_1,D_2\}\in Der(A)$ if and only if}\nonumber\\
				&& +2(D_{1}(u))\cdot(D_{2}(v))+2(D_{2}(u))\cdot(D_{1}(v))=0 \label{Der}.
			\end{eqnarray}
			\item Let $D_1\in Der(A)$ and $D_2\in Der(A).$ Then,\\
			$\bullet$ $\{D_1,D_2\}\in ADer(A)$ if and only if f (\ref{antider}) holds,\\
			$\bullet$ $\{D_1,D_2\}\in Der(A)$ if and only if (\ref{Der}) holds.\\
			\item Let $D_1\in ADer(A)$ and $D_2\in Der(A).$ Then,\\
			$\bullet$ $\{D_1,D_2\}\in ADer(A)$ if and only if f (\ref{Der}) holds,\\
			$\bullet$ $\{D_1,D_2\}\in Der(A)$ if and only if (\ref{antider}) holds.\\
		\end{enumerate}
	\end{proposition}
	\begin{proof}
		\begin{enumerate}
			\item 	Let $D_1,D_2\in ADer(A).$ Then, using (\ref{ader1}) for $D_1$ and $D_2$, we have for all \\ 
			$u,v\in A$,\\
			$\begin{array}{lll}
				\{D_1,D_2\}(u\cdot v)=D_{1}\bigg(D_{2}(u\cdot v)\bigg)+D_{2}\bigg(D_{1}(u\cdot v)\bigg)\\
				=D_{1}\bigg(-(D_{2}(u))\cdot v-u\cdot(D_{2}(v))\bigg)+D_{2}\bigg(-(D_{1}(u))\cdot v-u\cdot(D_{1}(v))\bigg)\\
				=-D_{1}\bigg((D_{2}(u))\cdot v\bigg)-D_{1}\bigg(u\cdot(D_{2}(v)) \bigg)-
				D_{2}\bigg((D_{1}(u))\cdot v\bigg)-D_{2}\bigg(u\cdot(D_{1}(v)) \bigg)\\
				=\bigg(D_{1}(D_{2}(u))\bigg)\cdot v+(D_{2}(u))\cdot(D_{1}(v))+(D_{1}(u))\cdot(D_{2}(v))+u\cdot\bigg(D_{1}(D_{2}(v))\bigg)
				\\+
				\bigg(D_{2}(D_{1}(u))\bigg)\cdot v+(D_{1}(u))\cdot(D_{2}(v))+(D_{2}(u))\cdot(D_{1}(v))+u\cdot\bigg(D_{2}(D_{1}(v))\bigg)\\
				=\bigg(D_{1}D_{2}(u)\bigg)\cdot v+\bigg(D_{2}D_{1}(u))\bigg)\cdot v+u\cdot\bigg(D_{1}D_{2}(v)\bigg)+u\cdot\bigg(D_{2}D_{1}(v)\bigg)\\
				+2(D_{1}(u))\cdot(D_{2}(v))+2(D_{2}(u))\cdot(D_{1}(v))\\
				=\bigg(\{D_{1},D_{2}\}(u)\bigg)\cdot v+u\cdot\bigg(\{D_{1},D_{2}\}(v)\bigg)+2(D_{1}(u))\cdot(D_{2}(v))+2(D_{2}(u))\cdot(D_{1}(v)).
			\end{array}  $\\
			Hence, $\{D_1,D_2\}\in ADer(A)$ (resp. $\{D_1,D_2\}\in Der(A)$ if and only if  (\ref{antider}) (resp. (\ref{Der}) ) holds.
			\item  Similarly, we prove (2) and (3).
		\end{enumerate}	
	\end{proof}
	\section{Cohomology of left pre-Jacobi-Jordan algebra}
	In this section, we introduce a cohomology theory for left pre-Jacobi-Jordan algebras fitting with deformation theory. First, we define the differential operators and then define the first and second cohomology groups in low degrees. Finally, interpretations of these groups are obtained.
	\subsection{Definition of cohomologies}
	Let $ (A,\cdot) $ be a left pre-Jacobi-Jordan algebra and $ (V;\rho,\mu) $ is its representation. Denote for all $ n\in \mathbb{N},n\geq 1 $, the set of $n$-linear maps from $A$ to $V$ by $\mathbf{C}^{n}(A,V)$ i.e.,  $\mathbf{C}^{n}(A,V):=\{f:A^{\otimes n}\rightarrow V, f \text{ is linear}\} $
	and $\mathbf{A}^{n}(A,V):=\{f\in Hom(\wedge^{n-1} A\otimes A, V), \circlearrowleft_{(u,v,w)} f(u\ast v,x_1,\cdots, x_{n-1},w\cdot x_n)=0, \forall (u,v,w, x_1,\cdots, x_n)\in A^{\otimes (n+3)}\}$.
	More precisely, $f\in \mathbf{A}^{n}(A,V)$ if only if $f\in Hom(\wedge^{n-1} A\otimes A, V)$ and 
	\begin{eqnarray}
		&& f(u\ast v, x_1,\cdots, x_{n-1}, w\cdot x_{n})+f(v\ast w, x_1,\cdots, x_{n-1}, u\cdot x_{n})\nonumber\\
		&&+f(w\ast u, x_1,\cdots, x_{n-1}, v\cdot x_{n})=0  \mbox{ $\forall (u,v,w, x_1,\cdots, x_n)\in A^{\otimes (n+3)}$.}\label{cA}
	\end{eqnarray}
	Next we define linear maps $  d^{n}: \mathbf{C}^{n}(A,V)\rightarrow \mathbf{C}^{n+1}(A,V) $ 
	and $\delta^n : \mathbf{A}^{n}(A,V)\rightarrow \mathbf{C}^{n+1}(A,V) $ by
	\begin{eqnarray}
		&& d^n f(x_1, \cdots, x_{n+1})\nonumber\\
		&&=\sum\limits_{i=1}^{n}\rho(x_i)f(x_1,\cdots,\widehat{x_i},\cdots, x_{n+1})
		+\sum\limits_{i=1}^{n}\mu(x_{n+1})f(x_1,\cdots,\widehat{x_i},\cdots, x_n,x_i)\nonumber\\
		&&+\sum\limits_{i=1}^{n}f(x_1,\cdots,\widehat{x_i},\cdots, x_n,x_i\cdot x_{n+1})
		+\sum\limits_{1\leq i< j\leq n} f(x_i\ast x_j, x_1, \cdots, \widehat{x_i}, \cdots, \widehat{x_j}, \cdots, x_{n+1}),
	\end{eqnarray}
	\begin{eqnarray}
		&& \delta^n f(x_1, \cdots, x_{n+1})\nonumber\\
		&&=\sum\limits_{i=1}^{n}\rho(x_i)f(x_1,\cdots,\widehat{x_i},\cdots, x_{n+1})
		+\sum\limits_{i=1}^{n}\mu(x_{n+1})f(x_1,\cdots,\widehat{x_i},\cdots, x_n,x_i)\nonumber\\
		&&-\sum\limits_{i=1}^{n}f(x_1,\cdots,\widehat{x_i},\cdots, x_n,x_i\cdot x_{n+1})
		-\sum\limits_{1\leq i< j\leq n} f(x_i\ast x_j, x_1, \cdots, \widehat{x_i}, \cdots, \widehat{x_j}, \cdots, x_{n+1}).
	\end{eqnarray}
	\begin{proposition}\label{Pro7}
		With the above notations,  for all $n\in\mathbb{N}$  with $n\geq 2,$ we have $d^n\circ \delta^{n-1}=0.$	
	\end{proposition}
	\begin{proof}
		Let $ n\in\mathbb{N}, \ n\geq 2,$ $ f\in \mathbf{A}^{n-1}(A,V)$ and $ x_{1},\cdots,x_{n+1}\in A $. We have \\\\
		$ \begin{array}{lll}
			&&
			d^{n}\delta^{n-1}f(x_{1},\cdots,x_{n+1})\\
			&&=\displaystyle\sum_{i=1}^{n}\rho(x_{i})\delta^{n-1}f(x_{1},\cdots,\widehat{x_{i}},\cdots,x_{n+1})
			+\displaystyle\sum_{i=1}^{n}\mu(x_{n+1})\delta^{n-1}f(x_{1},\cdots,\widehat{x_{i}},\cdots,x_{n},x_{i})\\
			&&+\displaystyle\sum_{i=1}^{n}\delta^{n-1}f(x_{1},\cdots,\widehat{x_{i}},\cdots,x_{n},x_{i}\cdot x_{n+1})
			+\displaystyle\sum_{1\leq i<j\leq n}\delta^{n-1}f(x_{i}\star x_{j},\cdots,\widehat{x_{i}},\cdots,\widehat{x_{j}},\cdots x_{n+1}).
		\end{array} $\\\\
		The first sum of the right hand side:\\
		$ \begin{array}{lll}
			&&
			\displaystyle\sum_{i=1}^{n}\rho(x_{i})\delta^{n-1}f(x_{1},\cdots,\widehat{x_{i}},\cdots,x_{n+1})\\	
			&&=\displaystyle\sum_{1\leq i<j\leq n}\bigg(\rho(x_{i})\circ \rho(x_{j})+\rho(x_{j})\circ \rho(x_{i})\bigg)f(x_{1},\cdots,\widehat{x_{i}},\cdots,\widehat{x_{j}},\cdots,x_{n+1})\\
			&&+\displaystyle\sum_{\substack{1\leq i,j\leq n\\ j\neq i }}\rho(x_{i})\circ\mu(x_{n+1})f(x_{1},\cdots,\widehat{x_{i}},\cdots,\widehat{x_{j}},\cdots,x_{n},x_{j})\\
			&&-\displaystyle\sum_{\substack{1\leq i,j\leq n\\ j\neq i }}\rho(x_{i})f(x_{1},\cdots,\widehat{x_{i}},\cdots,\widehat{x_{j}},\cdots,x_{n},x_{j}\cdot x_{n+1})\\
			&&-\displaystyle\sum_{\substack{1\leq i,j,k\leq n\\ j\neq i,k\neq i,j<k }}\rho(x_{i})f(x_{j}\ast x_{k},x_{1},\cdots,\widehat{x_{i}},\cdots,\widehat{x_{j}},\cdots,\widehat{x_{k}},\cdots , x_{n+1}).
		\end{array} $\\\\
		The second sum of the right hand side:\\
		$ \begin{array}{lll}
			&&	\displaystyle\sum_{i=1}^{n}\mu(x_{n+1})\delta^{n-1}f(x_{1},\cdots,\widehat{x_{i}},\cdots,x_{n},x_{i})\\
			&&	=\displaystyle\sum_{\substack{1\leq i,j\leq n\\ j\neq i }}\mu(x_{n+1})\circ \rho (x_{j})f(x_{1},\cdots,\widehat{x_{i}},\cdots,\widehat{x_{j}},\cdots,x_{i})\\
			&&+\displaystyle\sum_{\substack{1\leq i,j\leq n\\ j\neq i }}\mu(x_{n+1})\circ\mu(x_{i})f(x_{1},\cdots,\widehat{x_{i}},\cdots,\widehat{x_{j}},\cdots,x_{j})\\
			&&-\displaystyle\sum_{\substack{1\leq i<j\leq n\\ }}\mu(x_{n+1})f(x_{1},\cdots,\widehat{x_{i}},\cdots,\widehat{x_{j}},\cdots,,x_{i}\ast x_{j})\\
			&&-\displaystyle\sum_{\substack{1\leq i,j,k\leq n\\ j\neq i,k\neq i,j<k }}\mu(x_{n+1})f(x_{j}\ast x_{k},x_{1},\cdots,\widehat{x_{i}},\cdots,\widehat{x_{j}},\cdots,\widehat{x_{k}},\cdots , x_{i}).
		\end{array} $\\\\
		The third sum of the right hand side:\\
		$ \begin{array}{lll}
			&&\displaystyle\sum_{i=1}^{n}\delta^{n-1}f(x_{1},\cdots,\widehat{x_{i}},\cdots,x_{n},x_{i}\cdot x_{n+1}) \\
			&&=\displaystyle\sum_{\substack{1\leq i,j\leq n\\ j\neq i }}\rho (x_{j})f(x_{1},\cdots,\widehat{x_{i}},\cdots,\widehat{x_{j}},\cdots,x_{i}\cdot x_{n+1})\\
			&&+\displaystyle\sum_{\substack{1\leq i,j\leq n\\ j\neq i }}\mu(x_{i}\cdot x_{n+1})f(x_{1},\cdots,\widehat{x_{i}},\cdots,\widehat{x_{j}},\cdots,x_{n},x_{j})\\
			&&-\displaystyle\sum_{\substack{1\leq i<j\leq n\\ }}f(x_{1},\cdots,\widehat{x_{i}},\cdots,\widehat{x_{j}},\cdots,x_{n},x_{j}\cdot (x_{i}\cdot x_{n+1}))\\
			&&-\displaystyle\sum_{\substack{1\leq i<j\leq n\\ }}f(x_{1},\cdots,\widehat{x_{i}},\cdots,\widehat{x_{j}},\cdots,x_{n},x_{i}\cdot (x_{j}\cdot x_{n+1}))\\
			&&-\displaystyle\sum_{\substack{1\leq i,j,k\leq n\\ j\neq i,k\neq i,j<k }}f(x_{j}\ast x_{k},x_{1},\cdots,\widehat{x_{i}},\cdots,\widehat{x_{j}},\cdots,\widehat{x_{k}},\cdots , x_{n},x_{i}\cdot x_{n+1}).
		\end{array} $\\\\\\
		The first sum of the right hand side:\\
		$ \begin{array}{lll}
			&& \displaystyle\sum_{1\leq i<j\leq n}\delta^{n-1}f(x_{i}\ast x_{j},\cdots,\widehat{x_{i}},\cdots,\widehat{x_{j}},\cdots x_{n+1})\\			
			&&=\displaystyle\sum_{1\leq i<j\leq n}\rho (x_{i}\ast x_{j})f(x_{1},\cdots,\widehat{x_{i}},\cdots,\widehat{x_{j}},\cdots,x_{n+1})\\
			&&+\displaystyle\sum_{\substack{1\leq i,j,k\leq n\\ k\neq i,k\neq j,i<j }}\rho(x_{k})f(x_{i}\ast x_{j},x_{1},\cdots,\widehat{x_{i}},\cdots,\widehat{x_{j}},\cdots,\widehat{x_{k}},\cdots , x_{n}, x_{n+1})\\
			&&+\displaystyle\sum_{1\leq i<j\leq n}\mu (x_{n+1})f(x_{1},\cdots,\widehat{x_{i}},\cdots,\widehat{x_{j}},\cdots,x_{n},x_{i}\ast x_{j})\\
			&&+\displaystyle\sum_{\substack{1\leq i,j,k\leq n\\ k\neq i,k\neq j,i<j }}\mu(x_{n+1})f(x_{i}\ast x_{j},x_{1},\cdots,\widehat{x_{i}},\cdots,\widehat{x_{j}},\cdots,\widehat{x_{k}},\cdots , x_{n}, x_{k})\\
			&&-\displaystyle\sum_{1\leq i<j\leq n}f(x_{1},\cdots,\widehat{x_{i}},\cdots,\widehat{x_{j}},\cdots,x_{n},(x_{i}\ast x_{j})\cdot x_{n+1})\\
		\end{array} $\\
	$ \begin{array}{lll}
			&&-\displaystyle\sum_{\substack{1\leq i,j,k\leq n\\ k\neq i,k\neq j,i<j }}f(x_{i}\ast x_{j},x_{1},\cdots,\widehat{x_{i}},\cdots,\widehat{x_{j}},\cdots,\widehat{x_{k}},\cdots , x_{n}, x_{k}\cdot x_{n+1})\\
			&&-\displaystyle\sum_{\substack{1\leq i,j,k,l\leq n\\
					\{i,j\}\cap\{k,l\}=\emptyset\\ i<j,k<l
			}}f(x_{k}\ast x_{l},x_{i}\ast x_{j},x_{1},\cdots,\widehat{x_{i}},\cdots,\widehat{x_{j}},\cdots,\widehat{x_{k}},\cdots ,\widehat{x_{l}},\cdots, x_{n+1})\\
			&&-\displaystyle\sum_{\substack{1\leq i,j,k\leq n\\ k\neq i,k\neq j,i<j }}f((x_{i}\ast x_{j})\ast x_{k},x_{1},\cdots,\widehat{x_{i}},\cdots,\widehat{x_{j}},\cdots,\widehat{x_{k}},\cdots , x_{n+1})\\
			
			&&=\displaystyle\sum_{1\leq i<j\leq n}\rho (x_{i}\ast x_{j})f(x_{1},\cdots,\widehat{x_{i}},\cdots,\widehat{x_{j}},\cdots,x_{n+1})\\
			&&+\displaystyle\sum_{\substack{1\leq i,j,k\leq n\\ k\neq i,k\neq j,i<j }}\rho(x_{k})f(x_{i}\ast x_{j},x_{1},\cdots,\widehat{x_{i}},\cdots,\widehat{x_{j}},\cdots,\widehat{x_{k}},\cdots , x_{n}, x_{n+1})\\
			&&+\displaystyle\sum_{1\leq i<j\leq n}\mu (x_{n+1})f(x_{1},\cdots,\widehat{x_{i}},\cdots,\widehat{x_{j}},\cdots,x_{n},x_{i}\ast x_{j})\\
			&&+\displaystyle\sum_{\substack{1\leq i,j,k\leq n\\ k\neq i,k\neq j,i<j }}\mu(x_{n+1})f(x_{i}\ast x_{j},x_{1},\cdots,\widehat{x_{i}},\cdots,\widehat{x_{j}},\cdots,\widehat{x_{k}},\cdots , x_{n}, x_{k})\\
			&&-\displaystyle\sum_{1\leq i<j\leq n}f(x_{1},\cdots,\widehat{x_{i}},\cdots,\widehat{x_{j}},\cdots,x_{n},(x_{i}\ast x_{j})\cdot x_{n+1})\\
			&&-\displaystyle\sum_{\substack{1\leq i,j,k\leq n\\ k\neq i,k\neq j,i<j }}f(x_{i}\ast x_{j},x_{1},\cdots,\widehat{x_{i}},\cdots,\widehat{x_{j}},\cdots,\widehat{x_{k}},\cdots , x_{n}, x_{k}\cdot x_{n+1}),\\	
		\end{array} $\\\\
		since $f$ is skew-symmetric and by using Jacobi identity. Consequently,\\
		\\
		$	\begin{array}{lll}
			&& d^{n}\delta^{n-1}f(x_{1},\cdots,x_{n+1})\\
			&&=\displaystyle\sum_{1\leq i<j\leq n}\bigg(\rho(x_i\ast x_j)+\rho(x_{i})\circ \rho(x_{j})+\rho(x_{j})\circ \rho(x_{i})\bigg)f(x_{1},\cdots,\widehat{x_{i}},\cdots,\widehat{x_{j}},\cdots,x_{n+1})\\
			&&+\displaystyle\sum_{\substack{1\leq i,j\leq n\\ j\neq i }}\bigg(\mu(x_i\cdot x_{n+1})+\mu(x_{n+1})\mu(x_i)+\mu(x_{n+1})\rho(x_i)+\rho(x_{i})\mu(x_{n+1})\bigg)f(x_{1},\cdots,\\
			&&\widehat{x_{i}},\cdots,\widehat{x_{j}},\cdots,x_{n},x_{j})\\
			&&-\displaystyle\sum_{1\leq i<j\leq n}f\bigg(x_{1},\cdots,\widehat{x_{i}},\cdots,\widehat{x_{j}},\cdots,x_{n},(x_{i}\ast x_{j})\cdot x_{n+1}+x_i\cdot(x_j\cdot x_{n+1})+x_j\cdot(x_i\cdot x_{n+1})\bigg)\\
			&&-2\displaystyle\sum_{\substack{1\leq i,j,k\leq n\\ k\neq i,k\neq j,i<j }}f(x_{i}\ast x_{j},x_{1},\cdots,\widehat{x_{i}},\cdots,\widehat{x_{j}},\cdots,\widehat{x_{k}},\cdots , x_{n}, x_{k}\cdot x_{n+1})=0
		\end{array}	$\\\\
		by (\ref{7}), (\ref{8}), (\ref{6}) and (\ref{cA}).
	\end{proof}
	Consider linear maps $ d^0: \mathbf{C}^{0}(A,V)\rightarrow \mathbf{C}^{1}(A,V)$ and  $ d^0: \mathbf{A}^{0}(A,V)\rightarrow \mathbf{C}^{1}(A,V)$ by: \\
	$	\delta^{0}v(x)=d^{0}v(x):=\rho(x)v+\mu(x)v,  \forall (v,x)\in V\otimes A.$
	Then by straightforward computations, we have for all $v\in V$ and $(x,y)\in A^{\otimes 2}$:
	\begin{eqnarray}
		&&(d^{1}\circ \delta^{0}v)(x,y)=\rho(x)\delta^{0}v(y)+\mu(y)\delta^{0}v(x)+\delta^{0}v(x\cdot y)\nonumber\\
		&&=\rho(x)\circ\rho(y)v+\rho(x)\circ\mu(y)v+\mu(y)\circ\rho(x)v+\mu(y)\circ\mu(x)v+\rho(x\cdot y)v+\mu(x\cdot y)v\nonumber\\
		&&\overset{(\ref{8})}{=}\rho(x\cdot y)v+\rho(x)\circ\rho(y)v.\nonumber
	\end{eqnarray} 
	Hence, 
	\begin{eqnarray}
		(d^{1}\circ \delta^{0}v)(x,y)=0 \Longleftrightarrow \rho(x\cdot y)v+\rho(x)\circ\rho(y)v=0.
	\end{eqnarray}
	Therefore, we put $ \mathbf{A}^{0}(A,V)=\mathbf{C}^{0}(A,V):=\{v\in V,\rho(x\cdot y)v+\rho(x)\circ\rho(y)v=0 
	,\forall x,y\in A\}=V^{r.Aas}$.\\
	
	Associated to the representation $ (V;\rho,\mu) $, we obtain the complex $ (C(A,V),d,\delta) $. We denote the set of $k$-cocycles by $ \mathcal{Z}^{k}(A,V)$ and the set of  $k$-coboundaries by $ \mathcal{B}^{k}(A,V) $  where
	\begin{eqnarray}
		&& \mathcal{Z}^{k}(A,V)=Ker(d^{k}) \label{18}, \,\ k\in\mathbb{N},\\ 
		&& \mathcal{B}^{k}(A,V)=Im(\delta^{k-1}),\,\ k\in\mathbb{N}\setminus\{0\}\label{18'},\\
		&&\mathcal{B}^{0}(A,V)=0.\label{18"}
	\end{eqnarray}
	The $k$-th cohomology group $\mathcal{H}^{k}(A,V)$ is given by:
	\begin{eqnarray}
		\label{10}
		\mathcal{H}^{k}(A,V)=\mathcal{Z}^{k}(A,V)/\mathcal{B}^{k}(A,V) ,\,\,\ k\in\mathbb{N}.
	\end{eqnarray}
	\subsection{Low degree cohomology spaces}
	In this subsection, we give an interpretation of the low degree cohomology spaces in terms of
	algebraic properties of the pre-Jacobi-Jordan algebra. In degree zero we have:
	\begin{eqnarray}
		\mathcal{H}^{0}(A,V)&=&\{v\in C^0(A,V), d^0v=0\}\nonumber\\
		&=&\{v\in V,\rho(x)v+\mu(x)v =\rho(x\cdot y)v+\rho(x)\circ\rho(y)v=0, \forall x,y\in A\} \nonumber\\
		&=&V^{inv}\cap V^{r.Aas}.\nonumber
	\end{eqnarray}
	For example if $V=A$ is the adjoint representation of $(A,\cdot),$ we obtain
	\begin{eqnarray}
		\mathcal{H}^{0}(A,V)
		&=& \{v\in V,x\cdot v+v\cdot x =(x\cdot y)\cdot v+x\cdot(y\cdot v)=0, \forall x,y\in A\}\nonumber\\
		&=&A^{inv}\cap A^{r.Aas}.\nonumber
	\end{eqnarray}
	Now, let investigate the first group of cohomology. We have:
	\begin{eqnarray}
		\mathcal{H}^{1}(A,V)&=&\mathcal{Z}^{1}(A,V)/\mathcal{B}^{1}(A,V),
	\end{eqnarray}
	where
	\begin{eqnarray}
		\mathcal{Z}^{1}(A,V)&=&\{f\in C^1(A,V), f(x_1\cdot x_2)=-\rho(x_1)f(x_2)-\mu(x_2)f(x_1), \forall (x_1,x_2)\in A^{\otimes 2}\}\nonumber\\
		&=& ADer(A,V),\nonumber\\
		\mathcal{B}^{1}(A,V)&=& \{f\in C^1(A,V), \exists v\in V, f(x)=\rho(x)v+\mu(x)v, 
		\rho(x\cdot y)v+\rho(x)\rho(y)v=0,\nonumber\\ 
		&& \forall (x,y)\in A^{\otimes 2}\}\nonumber\\
		&=&IADer(A,V).\nonumber
	\end{eqnarray}
	Therefore $$\mathcal{H}^{1}(A,V)=OADer(A,V) \mbox{  (the space of outer-antiderivations of A with values in V).}$$
	A particular case is $V=A$ be the adjoint representation $(A,L, R),$ then
	\begin{eqnarray}
		\mathcal{H}^{1}(A,A)&=&\mathcal{Z}^{1}(A,A)/\mathcal{B}^{1}(A,A),
	\end{eqnarray}
	where
	\begin{eqnarray}
		\mathcal{Z}^{1}(A,A)&=&\{f\in C^1(A,A), f(x_1,x_2)=-x_1\cdot f(x_2)-f(x_1)\cdot x_2, \forall (x_1,x_2)\in A^{\otimes 2}\}\nonumber\\
		&=& ADer(A),\nonumber\\
		\mathcal{B}^{1}(A,A)&=& \{f\in C^1(A,A), \exists v\in A, f(x)=x\cdot v+v\cdot x, 
		(x\cdot y)\cdot v+x\cdot(y\cdot v)=0,\nonumber\\ 
		&& \forall (x,y)\in A^{\otimes 2}\}\nonumber\\
		&=&IADer(A).\nonumber
	\end{eqnarray}
	Hence, $$\mathcal{H}^{1}(A)=OADer(A) \mbox{  (the space of outer-antiderivations of A).}$$
Another very interesting case is $V=\mathbb{K}$ with the trivial action. Here
$$\mathcal{H}^1(A,\mathbb{K})=\{f\in Hom(A,\mathbb{K}), f(x\cdot y)=0,\ \forall x,y\in A\}/\{0\}=(A/A^2)^*.$$
\begin{example}\label{ex1} (Computation of $\mathcal{H}^1(\mathcal{A}_1,\mathcal{A}_1)$). 
Here, we investigate the first group of cohomology
of $\mathcal{A}_1$ with the adjoint representation. We find the dimension and generators of $\mathcal{H}^1(\mathcal{A}_1,\mathcal{A}_1).$

A linear map $D$ on the previous algebra $\mathcal{A}_1$ is an antiderivation if and only if the matrix of $D$ on the basis $\{e_1, e_2\}$ is given by
$$\left(
\begin{array}{cc}
	a_1&0\\
	a_2&-2a_1
\end{array}
\right),$$ where, $a_1,\ a_2$ are scalars.\\

Observe that any $w=ae_1+be_2 \in \mathcal{A}_1$ where $a$ and $b$ are scalars, satisfies (\ref{id1}). Hence, the matrix of the inner antiderivation $D_w:= R_w+L_w:\mathcal{A}_1\rightarrow \mathcal{A}_1$ defined by $D_w (u):= u\cdot w+w\cdot u;$ $\forall u\in \mathcal{A}_1$ has the following form
$$\left(
\begin{array}{cc}
	0&0\\
	2b&0
\end{array}
\right).$$
 Thus, the dimension of $\mathcal{H}^1(\mathcal{A}_1,\mathcal{A}_1)$ is $1$ and elements of $\mathcal{H}^1(\mathcal{A}_1,\mathcal{A}_1)$
can be written as the classes of the following matrices:
$$\left(
\begin{array}{cc}
	a_1&0\\
	0&-2a_1
\end{array}
\right).$$
\end{example}
\begin{example}(Computation of $\mathcal{H}^1(\mathcal{A}_2,\mathbb{K})$). 
	Observe that $\varphi$ belongs to $\mathcal{H}^1(\mathcal{A}_2,\mathbb{K})=\{f\in Hom(A,\mathbb{K}), f(x\cdot y)=0,\ \forall x,y\in \mathcal{A}_2\}$ if and only if the matrix of $\varphi$ with respect to the basis $\{e_1, e_2, e_3, e_4 \}$ is of the form  $(a\,\ 0\,\ b\,\ 0),\,\ a,b\in \mathbb{K}$. Thus, $\mathcal{H}^1(\mathcal{A}_2,\mathbb{K})$ is spanned by $\{\varphi_1:=(1\,\ 0\,\ 0\,\ 0); \varphi_2:=(0\,\ 0\,\ 1\,\ 0)\}.$
\end{example}
\begin{example} (Computation of $\mathcal{H}^1(\mathcal{A}_2,\mathcal{A}_2)$). 
	Similarly as Example \ref{ex1}, we investigate the first group of cohomology
	of $\mathcal{A}_2$ with the adjoint representation. We find the dimension and generators of $\mathcal{H}^1(\mathcal{A}_2,\mathcal{A}_2).$
	
	A linear map $D$ on the previous algebra $\mathcal{A}_1$ is an antiderivation if and only if the matrix of $D$ on the basis $\{e_1, e_2, e_3,e_4\}$ is given by
	$$\left(
	\begin{array}{cccc}
		a_1&0&0&0\\
		a_2&-2a_1&c_2&0\\
		a_3&0&c_3&0\\
		a_4&-2a_3&c_4&-a_1-c_3
	\end{array}
	\right),$$ where, $a_1,\ a_2,\ a_3,\ a_4,\ c_2,\ c_3,\ c_4$ are scalars.\\
	
	Observe that any $w=ae_1+be_2+ce_3+de_4 \in \mathcal{A}_2$ where $a,\ b,\ c,\ d$  are scalars, satisfies (\ref{id1}). Hence, the matrix of the inner antiderivation $D_w:= R_w+L_w:\mathcal{A}_2\rightarrow \mathcal{A}_2$ defined by $D_w (u):= u\cdot w+w\cdot u;$ $\forall u\in \mathcal{A}_2$ has the following form
	$$\left(
	\begin{array}{cccc}
		0&0&0&0\\
		a&0&0&0\\
		0&0&0&0\\
		c&0&a&0
	\end{array}
	\right)$$
	where $a,\ c$ are scalars. Thus, the dimension of $\mathcal{H}^1(\mathcal{A}_2,\mathcal{A}_2)$ is $5$ and elements of $\mathcal{H}^1(\mathcal{A}_2,\mathcal{A}_2)$
	can be written as the classes of the following matrices:
	$$\left(
	\begin{array}{cccc}
		a_1&0&0&0\\
		a_2&-2a_1&c_2&0\\
		a_3&0&c_3&0\\
		0&-2a_3&0&-a_1-c_3
	\end{array}
	\right).$$ 
\end{example}
	\section{Linear deformations of pre-Jacobi-Jordan algebras }
	In this section, we study  a deformation theory for pre-Jacobi-Jordan algebras and show that the cohomology introduced in the previous section fits with formal deformations of left pre-Jacobi-Jordan algebras.
	\begin{definition}\label{D11}
		Let $(A,\cdot) $ be a left pre-Jacobi-Jordan algebra and $ \omega : A\times A\to A $ be a linear map. If for any $ t\in \mathbb{C} $, the multiplication $ \cdot_{t} $ defined by 
		\begin{eqnarray}
			\label{23}
			x\cdot_{t}y=x\cdot y+t\omega(x,y),\forall x,y\in A,
		\end{eqnarray}
		also gives a left pre-Jacobi-Jordan algebra structure, we say that $ \omega $ generates a linear deformation of the left pre-Jacobi-Jordan algebra $ (A,\cdot) $.	
	\end{definition}
	Since  is  $ x\cdot_{t}y=x\cdot y+t\omega(x,y)$ for all $ t\in \mathbb{C} $, we deduce that $ \omega $ generates a lineair deformation of the left pre-Jacobi-Jordan algebra $ (A,\cdot) $ if and only if for any $ x,y\in A $,
	\begin{eqnarray}
		\label{24}
		\begin{array}{lll}
			\omega(x,y)\cdot z+\omega(y,x)\cdot z
			+\omega(x\cdot y,z)+\omega(y\cdot x,z)+\omega(x,y\cdot z)+\omega(y,x\cdot z)\\
			x\cdot \omega(y,z)+y\cdot \omega(x,z)=0,
		\end{array}
	\end{eqnarray}
	\begin{eqnarray}
		\label{25}
		\omega(\omega(x,y),z)+\omega(x,\omega(y,z))+\omega(\omega(y,x),z)+\omega(y,\omega(x,z))=0.
	\end{eqnarray}
	Note that Eq.(\ref{24}) means that $ \omega $ is a $ 2 $-cocycle of the regular representation of left pre-Jacobi-Jordan algebra $ (A,\cdot) $ and Eq. (\ref{25}) means that $ (A,\omega) $ is a left pre-Jacobi-Jordan algebra. \\
	
	Hence, we have the following proposition.
	\begin{proposition}\label{an1}
	Let $ (A,\mu) $ be a left pre-Jacobi-Jordan algebra and $\mu_t:=\mu +t\omega$ be a linear deformation of $A$. Then 
	$$\omega\in\mathcal{Z}^2(A,A).$$
	\end{proposition}
	\begin{definition}\label{D14} Let $ (A,\cdot) $ be a left pre-Jacobi-Jordan algebra and $ N : A\to A$ be a linear map.\\
		Two deformations $ A_{t}=(A,\cdot_{t}) $ and $ A'_{t}=(A,\cdot'_{t}) $ of the pre-Jacobi-Jordan algebra $ (A,\cdot) $ generated respectively by $ \omega $ and $ \omega' $, are said  equivalent if there exists a family of left pre-Jacobi-Jordan algebra  morphisms $ Id_{A}+tN : A_{t}\to A'_{t} $ with $t\in\mathbb{C}$. A deformation is said trivial if there exists a family of pre-Jacobi-Jordan algebra  morphisms $ Id_{A}+tN : A_{t}\to A $ 
	\end{definition}
	Let $ Id_{A}+tN : A_{t}\to A'_{t} $ be a left pre-Jacobi-Jordan algebra  morphism. Then, we have :\\
	$
	(Id_{A}+tN )(x\cdot_{t} y)=(x+tN(x))\cdot'_{t}(y+tN(y)),\forall x,y\in A  $. Hence,\\
	$ \begin{array}{lll}
		\bigg(x\cdot N(y)+N(x)\cdot y+\omega'(x,y)-\omega(x,y)-N(x\cdot y)\bigg)t\\
		+\bigg(N(x)\cdot N(y)+\omega'(x,N(y))+\omega'(N(x),y)-N(\omega(x,y))\bigg)t^{2}+
		\omega'(N(x),N(y))t^{3}=0.
	\end{array} $.\\ 
	It follows that
	\begin{eqnarray}
		\label{49}
		\omega'(x,y)-\omega(x,y)=-N(x)\cdot y-x\cdot N(y)+N(x\cdot y),\\
		\label{50}
		N(\omega(x,y))=N(x)\cdot N(y)+\omega'(N(x),y)+\omega'(x,N(y)),\\
		\label{56}
		\omega'(N(x),N(y))=0.
	\end{eqnarray}
	Eq.(\ref{49}) means that $ \omega-\omega'=\delta N $ and Eq.(\ref{56}) means that $ \omega'_{/Im(N)}=0$.\\\\
	Thus, we have the following result:
	\begin{proposition}\label{an2}
		If two deformations  $A_{t}=(A,\cdot_{t}) $ and $ A'_{t}=(A,\cdot'_{t}) $ of the pre-Jacobi-Jordan algebra $ (A,\cdot) $, generated by $ \omega $ and $ \omega' $, respectively, are equivalent, then $ \omega $ and $ \omega' $ are cohomologous i.e., are in the same cohomology class of $ \mathcal{H}^{2}(A,A) $.  	
	\end{proposition}
By the previous proposition, we observe that the  term $\omega $ of a linear deformation depends only on its cohomology class with respect
to zigzag cohomology of a pre-Jacobi-Jordan algebra.\\

Thanks to Proposition \ref{an1} and Proposition \ref{an2}, it is easy to prove:
\begin{proposition}
Let $ (A,\mu) $ be a left pre-Jacobi-Jordan algebra. There is, over $K[[t]]/t^2$ , a one-to-one
correspondence between elements $\omega$ of $\mathcal{H}^2(A,A)$ and linear deformations of $A$ defined by 
$$\mu_t:=\mu +t\omega.$$
\end{proposition}
	\section{Nijenhuis Operators on left Pre-Jacobi-Jordan}
	In this section, we introduce the notion of Nijenhuis operators on pre-Jacobi-Jordan algebras. Among other results, we show that a Nijenhuis operator  generates a trivial deformation of a given pre-Jacobi-Jordan 
	algebra.\\
	\\
	Let  consider a trivial deformation of a left pre-Jacobi-Jordan algebra $ (A,\cdot) $.\\\\
	Eqs. (\ref{49})-(\ref{56}) reduce to 
	\begin{eqnarray}
		\label{51}
		\omega(x,y)=N(x)\cdot y+x\cdot N(y)-N(x\cdot y),\\
		\label{52}
		N(\omega(x,y))=N(x)\cdot N(y).
	\end{eqnarray}
	Hence, Eqs.(\ref{51}) and (\ref{52}) give a meaning to the following definition.
	\begin{definition}\label{D13}
		Let $ (A,\cdot) $ be a left pre-Jacobi-Jordan algebra. A linear map $ N : A\to A $ is called a Nijenhuis operator if 
		\begin{eqnarray}
			\label{47}
			N(x)\cdot N(y)=N( x\cdot_N y),&\forall x,y\in A,
		\end{eqnarray}
		$$ where\,\,\,\
		x\cdot_N y:=N(x)\cdot y+x\cdot N(y)-N(x\cdot y)\,\ \forall x,y\in A.$$
	\end{definition}
	\begin{remark}
		Note that a Rota-Baxter operator of weight $-1$ on left pre-Jacobi-Jordan algebra $ A $ is exactly a Nijenhuis operator.
	\end{remark}
\begin{proposition}\label{pJJN}
	Let $ N : A\to A $ be a Nijenhuis operator on a left pre-Jacobi-Jordan algebra $ (A,\cdot) $. Then , $A_N:=(A,\cdot_{N})$ is left pre-Jacobi-Jordan algebra and  $ N $ is a morphism from  $ A_N $ to the initial left pre-Jacobi-Jordan algebra $ (A,\cdot) $.  	
\end{proposition}
\begin{proof}
Pick $ x,y,z\in A $ and denote by $Aasso_N$ the anti-associator induced by $\cdot_N,$ then we have\\
$ \begin{array}{lll}
	Aasso_{N}(x,y,z)+Aasso(y,x,z)\\
	\overset{(\ref{1})}{=}(x\cdot_{N} y)\cdot_{N} z+(y\cdot_{N} x)\cdot_{N} z+x\cdot_{N} (y\cdot_{N} z)+y\cdot_{N} (x\cdot_{N} z)\\
	\overset{(\ref{47})}{=}
	N(x\cdot_{N} y)\cdot z+(x\cdot_{N} y)\cdot N(z)-N((x\cdot_{N} y)\cdot z)+
	N(y\cdot_{N} x)\cdot z+(y\cdot_{N} x)\cdot N(z)\\
	-N((y\cdot_{N} x)\cdot z)
	+N(x)\cdot(y\cdot_{N} z)+x\cdot N(y\cdot_{N} z)-N(x\cdot(y\cdot_{N} z))+
	+N(y)\cdot(x\cdot_{N} z)\\
	+y\cdot N(x\cdot_{N} z)-N(y\cdot(x\cdot_{N} z))\\
	\overset{(\ref{57}),(\ref{47})}{=}	\bigg(N(x)\cdot N(y)\bigg)\cdot z+\bigg(N(x)\cdot y+x\cdot N(y)-N(x\cdot y)\bigg)\cdot N(z)\\
	-N\bigg(\bigg(N(x)\cdot y+x\cdot N(y)-N(x\cdot y)\bigg)\cdot z\bigg)
	+\bigg(N(y)\cdot N(x)\bigg)\cdot z\\
	+\bigg(N(y)\cdot x+y\cdot N(x)-N(y\cdot x)\bigg)\cdot N(z)
	-N\bigg(\bigg(N(y)\cdot x+y\cdot N(x)-N(y\cdot x)\bigg)\cdot z\bigg)\\
	+N(x)\cdot \bigg(N(y)\cdot z+y\cdot N(z)-N(y\cdot z)\bigg)+x\cdot \bigg(N(y)\cdot N(z)\bigg)\\
	-N\bigg(x\cdot \bigg(N(y)\cdot z+y\cdot N(z)-N(y\cdot z)\bigg)\bigg)  	
	+N(y)\cdot \bigg(N(x)\cdot z+x\cdot N(z)-N(x\cdot z)\bigg)\\
	+y\cdot \bigg(N(x)\cdot N(z)\bigg)
	-N\bigg(y\cdot \bigg(N(x)\cdot z+x\cdot N(z)-N(x\cdot z)\bigg)\bigg)\\
	=\bigg(N(x)\cdot N(y)\bigg)\cdot z+N(x)\cdot\bigg( N(y)\cdot z\bigg)+\bigg(N(y)\cdot N(x)\bigg)\cdot z+N(y)\cdot\bigg( N(x)\cdot z\bigg)\\
	\bigg(N(x)\cdot y\bigg)\cdot N(z)+N(x)\cdot\bigg( y\cdot N(z)\bigg)+\bigg(y\cdot N(x)\bigg)\cdot N(z)+y\cdot\bigg( N(x)\cdot N(z)\bigg)\\
	\bigg(x\cdot N(y)\bigg)\cdot N(z)+x\cdot\bigg(  N(y)\cdot N(z)\bigg)+\bigg( N(y)\cdot x\bigg)\cdot N(z)+ N(y)\cdot\bigg( x\cdot N(z)\bigg)\\
	-N(x\cdot y)\cdot N(z)-N(y\cdot x)\cdot N(z)-N(x)\cdot N(y\cdot z)-N(y)\cdot N(x\cdot z)\\
	-N\bigg(\bigg(N(x)\cdot y+x\cdot N(y)-N(x\cdot y)\bigg)\cdot z\bigg)-N\bigg(\bigg(N(y)\cdot x+y\cdot N(x)-N(y\cdot x)\bigg)\cdot z\bigg)\\
	-N\bigg(x\cdot\bigg(N(y)\cdot z+y\cdot N(z)-N(y\cdot z)\bigg)\bigg)  
	-N\bigg(y\cdot\bigg(N(x)\cdot z+x\cdot N(z)-N(x\cdot z)\bigg)\bigg)\\
	\mbox{\,\,\,\,\,\,\,\,\,\,\,\,\,\,\,\,\,\,\,\,\,\,\,\,\,\,\,\,\,\,\,\,\,\,\,\,\,\,\,\,\,\ (rearranging terms)}\\
\end{array} $\\
$ \begin{array}{lll}
	\overset{(\ref{47})}{=}Aasso_{N}\bigg(N(x),N(y),z\bigg)+Aasso_{N}\bigg(N(y),N(x),z\bigg)\\
	Aasso_{N}\bigg(N(x),y,N(z)\bigg)+Aasso_{N}\bigg(y,N(x),N(z)\bigg)\\
	Aasso_{N}\bigg(x,N(y),N(z)\bigg)+Aasso_{N}\bigg(N(y),x,N(z)\bigg)\\
	-N\bigg(\bigg(N(x\cdot y)\cdot z+(x\cdot y)\cdot N(z)-N((x\cdot y)\cdot z)\bigg)\bigg)-N\bigg(\bigg(N(y\cdot x)\cdot z+(y\cdot x)\cdot N(z)\\
	-N((y\cdot x)\cdot z)\bigg)\bigg)
	-N\bigg(N(x)\cdot(y\cdot z)+x\cdot N(y\cdot z)-N(x\cdot(y\cdot z))\bigg)-N\bigg(N(y)\cdot(x\cdot z)\\
	+y\cdot N(x\cdot z)-N(y\cdot(x\cdot z))\bigg)	
	-N\bigg(\bigg(N(x)\cdot y+x\cdot N(y)-N(x\cdot y)\bigg)\cdot z\bigg)\\
	-N\bigg(\bigg(N(y)\cdot x+y\cdot N(x)-N(y\cdot x)\bigg)\cdot z\bigg)
	-N\bigg(x\cdot\bigg(N(y)\cdot z+y\cdot N(z)-N(y\cdot z)\bigg)\bigg)\\
	-N\bigg(y\cdot\bigg(N(x)\cdot z+x\cdot N(z)-N(x\cdot z)\bigg)\bigg)\\
	\overset{(\ref{6})}{=}-N\bigg(Aass_{N}(x,y,N(z))+Aass_{N}(y,x,N(z))\bigg)-N\bigg(Aass_{N}(N(x),y,z))+\\Aass_{N}(y,N(x),z)\bigg)
	-N\bigg(Aass_{N}(N(y),x,z)+Aass_{N}(x,N(y),z)\bigg)\\-N^{2}\bigg(Aass_{N}(x,y,z)+Aass_{N}(y,x,z)\bigg)
	\overset{(\ref{1})}{=}0.
\end{array} $\\\\
Therefore, $ (A,\cdot_{N}) $ is left pre-Jacobi-Jordan algebra.
\end{proof}
	The following result says that a Nijenhuis operator gives rise to a trivial deformation of a left pre-Jacobi-Jordan.
	\begin{theorem}
		Let $ N $ be a Nijenhuis operator on the left pre-Jacobi-Jordan $ (A,\cdot) $. Then a deformation of $ (A,\cdot) $ can be obtained by putting
		\begin{eqnarray}
			\label{57}
			\omega(x,y)=(\delta^1 N)(x,y):=x\cdot_N y,&\forall x,y\in A.
		\end{eqnarray}
		Furthermore, this deformation is trivial.
	\end{theorem}
	\begin{proof}
		By the definition of $\omega$, it is clear that $d^2\omega=0$ since $d^2\circ \delta^1=0.$\\
		Therefore $ \omega $ is a $ 2- $cocycle of the regular representation of left pre-Jacobi-Jordan $ (A,\cdot)$.
		Next,  $\omega$ satifies Eq.(\ref{25}) since $(A,\cdot_N)$ is a pre-Jacobi-Jordan algebra by Proposition \ref{pJJN}. Finally, by straightforward computations, we have
		\begin{eqnarray}
			&&(Id+tN)(x)\cdot (Id+tN)(y)=x\cdot y+t(N(x)\cdot y+x\cdot N(y))+t^2 N(x)\cdot N(y),\nonumber
		\end{eqnarray}	
		\begin{eqnarray}
			&&(Id+tN)(x\cdot_{t}y)=x\cdot y+t(N(x)\cdot y+\omega(x,y))+t^2N(\omega(x,y))\nonumber\\
			&&	\overset{(\ref{57})}{=} x\cdot y+t(N(x)\cdot y+x\cdot N(y))+t^2N(N(x)\cdot y+x\cdot N(y)-N(x\cdot y))\nonumber\\
			&&
			\overset{(\ref{47})}{=} x\cdot y+t(N(x)\cdot y+x\cdot N(y))+t^2N(x)\cdot N(y).\nonumber
		\end{eqnarray}	
		It follows that
		\begin{eqnarray}
			&&(Id+tN)(x\cdot_{t}y)=(Id+tN)(x)\cdot (Id+tN)(y).\nonumber
		\end{eqnarray}
		Hence, this deformation is trivial.	
	\end{proof}	
Before giving the last result of this paper, let recall the following.
\begin{definition}
	Let $(A, \ast)$ be a Jacobi-Jordan algebra.
	A linear operator $N: A\rightarrow A$ is called a Nijienhuis operator if
	\begin{eqnarray}
		N(u)\ast N(v)=N(N(u)\ast v+u\ast N(v)-N(u\ast v)). \label{opn}
	\end{eqnarray}
\end{definition}
	\begin{proposition}\label{Pro10}
		Let $ N $ be a Nijenhuis operator on a pre-Jacobi-Jordan algebra $ (A,\cdot) $. Then $ N $ is also a Nijenhuis operator on the sub-adjacent Jacobi-Jordan $ A^{C} $.
	\end{proposition}
	\begin{proof}
		For all $ x,y\in A $, we have : \\
		$ \begin{array}{lll}
			N(x)\ast N(y)=N(x)\cdot N(y)+N(y)\cdot N(x)\\
			\overset{(\ref{47})}{=}N\bigg(N(x)\cdot y+x\cdot N(y)-N(x\cdot y)\bigg)+N\bigg(N(y)\cdot x+y\cdot N(x)-N(y\cdot x)\bigg)\\
			=N\bigg(N(x)\ast y+x\ast N(y)-N(x\ast y)\bigg).
		\end{array}
		$\\
		Therefore, $ N $ is a Nijenhuis operator on the sub-adjacent Jacobi-Jordan $ A^{C} $. 
	\end{proof}
	\begin{proposition}\label{Pro11}
		Let $ (A,\cdot) $ be a pre-Jacobi-Jordan algebra and $ N : A\to A $ be a linear map.
		\begin{enumerate}
			\item[(i)]If $ N $ is a Nijenhuis operator on $A$ then, for all $ \lambda\in \mathbb{K} $, $ N+\lambda Id_{A} $ is a Nijenhuis operator on $ A $.
			\item[(ii)]	If $ N^{2}=0 $, then $ N $ is a Nijenhuis operator  on $A$if and only if $ N $ is Rota-Baxter operator of weight zero on $ A $.
			\item[(iii)] If $ N^{2}=N $, then $ N $ is a Nijenhuis operator on $A$ if and only if $ N $ is Rota-Baxter operator of weight $ -1 $ on $ A $.  
		\end{enumerate}
	\end{proposition}
	\begin{proof}
		First, we prove (i) as follows. Suppose that $N$ is a Nijenhuis operator on $A$ and  let $\lambda\in \mathbb{K}$. Then, for all $ x,y\in A $, we compute \\
		$ \begin{array}{lll}
			(N+\lambda Id_{A})(x)\cdot(N+\lambda Id_{A})(y)-(N+Id_{A})\bigg((N+\lambda Id_{A})(x)\cdot y+x\cdot(N+\lambda Id_{A})(y)\\
			-(N+\lambda Id_{A})(x\cdot y)\bigg)=N(x)\cdot N(y)-N\bigg(N(x)\cdot y+x\cdot N(y)-N(x\cdot y)\bigg)+\lambda x\cdot N(y)\\
			+\lambda N(x)\cdot y+\lambda^{2}x\cdot y-\lambda\bigg(N(x)\cdot y+x\cdot N(y)-N(x\cdot y)+\lambda x\cdot y\bigg)-\lambda N(x\cdot y)\\
			=N(x)\cdot N(y)-N\bigg(N(x)\cdot y+x\cdot N(y)-N(x\cdot y)\bigg)\\
			=0 \text{  (since N is a Nijenhuis operator)}.	
		\end{array} $\\\\
		Hence, (i) holds.
		If observing that $ N(x)\cdot N(y)=N\bigg(N(x)\cdot y+x\cdot N(y\bigg)-N^{2}(x\cdot y) $  then  (ii) and (iii) follow.	
	\end{proof}	
	
\end{document}